\documentclass{amsart}
\usepackage{amsmath,amsfonts,amsthm,amssymb}
\usepackage{ifthen}
\usepackage{longtable}
\usepackage{array}
\usepackage{url}
\usepackage{multirow}

\newcommand{\Z}{\mathbb{Z}}
\newcommand{\Q}{\mathbb{Q}}
\newcommand{\R}{\mathbb{R}}
\newcommand{\C}{\mathbb{C}}
\newcommand{\N}{\mathbb{N}}

\newtheorem{theorem}{Theorem}
\newtheorem{lemma}{Lemma}
\newtheorem{corollary}{Corollary}
\newtheorem{proposition}{Proposition}

\theoremstyle{remark}

\begin{document}
\title[Characterization of the Fibonacci sequence]{On Arithmetic Progressions in Recurrences -\\ A new characterization of the Fibonacci sequence}
 \keywords{}

\author[\'{A}. Pint\'{e}r]{\'{A}kos Pint\'{e}r}
\thanks{ Research supported in part by the Hungarian Academy of
Sciences, OTKA grants T67580, K75566, and by the J\'anos Bolyai
Fellowship (\'A.P.), the Ervin Schr\"odinger Fellowship (V.Z.) and
the Foundation of Action Austria-Hungary, No. 75\"ou1}

\address{\'{A}. Pint\'{e}r \newline
         \indent Institute of Mathematics \newline
          \indent University of Debrecen \newline
         \indent H-4010 Debrecen, Hungary}
\email{apinter\char'100math.unideb.hu}

\author[V. Ziegler]{Volker Ziegler}
%\thanks{The author gratefully acknowledges support from the
%Austrian Science Fund (FWF) under project Nr. S9603}
\address{V. Ziegler \newline
         \indent Institute for Analysis and\newline \indent Computational Number Theory,\newline
          \indent Graz University of Technology \newline
         \indent Steyrergasse 30/IV, \newline
         \indent A-8010 Graz, Austria}
\email{ziegler\char'100finanz.math.tugraz.at}

\begin{abstract}
We show that essentially the Fibonacci sequence is the unique
binary recurrence which contains infinitely many three-term
arithmetic progressions. A criterion for general linear
recurrences having infinitely many three-term arithmetic
progressions is also given.
\end{abstract}

\maketitle

\section{Introduction}
Linear recurrence sequences have been studied since ancient times
and also in the last century the interest in recurrences was high.
Equations concerning linear recurrences have an extremely rich
literature. For instance, let $f_n$ be a recurrence sequence, then
the equation $f_n=0$ has been studied by several mathematicians.
The finiteness of zero-multiplicity of $f_n$ was proved by Skolem,
Mahler and Lech \cite{Skolem:1934,Mahler:1935,Lech:1953}. Although
one can give upper bounds for the number of solutions (see
\cite{Schmidt:1999}) in general, it is not possible to find all
solutions effectively. However, in the case of binary and ternary
recursions Mignotte \cite{Mignotte:1975} found effective growth
estimates and therefore in these cases we can give all values of
$n$ with $f_n=0$. But also equations of the type $Af_n=Bg_m$ were
studied by various authors (e.g. Schlickewei and Schmidt
\cite{Schlickewei:2000}). For a full account to study similar
linear equations in recurrence sequences we recommend the book of
Everest et. al. \cite{Everest:RS}.

Let $K$ be an algebraically closed field of characteristic zero,
$\Gamma$ a multiplicative subgroup of $K^*$ having finite rank
$r$, let $\mathcal A$ be a finite set of $t$-tuples $\in K^t$
having $n$ elements and put
\[H_t(\Gamma,\mathcal A)=\left\{ \sum_{i=1}^t a_i x_i \: :\: (a_1,\ldots,a_t)\in\mathcal A,
(x_1,\ldots,x_t)\in \Gamma^t\right\}.\] Hajdu \cite{Hajdu:2007}
proved that there exists a constant $C(r,t,n)$ such that there
exists no non-constant arithmetic progression in
$H_t(\Gamma,\mathcal A)$ with length $\geq C(r,t,n)$. A direct
consequence of this result is that the length of arithmetic
progressions in simple recurrence sequences is bounded by a
constant depending only on its order $d$. Recently, the interest
in arithmetic progressions in certain number-theoretical
structures, like the points on elliptic curves
\cite{Bremner:2000}, solutions of Pellian equations
\cite{Pethoe:2008, Dujella:2008} or norm form equations
\cite{Berczes:2008}, has increased. We also mention here a new
result due to by Schwartz, Solymosi and de Zeeuw
\cite{Schwartz:2009}.

The purpose of this paper is to connect to these investigations.
Roughly speaking we show that sequences that contain
infinitely many three-term arithmetic progressions are very
special. Note that finding non-trivial three term arithmetic
progressions $(f_m,f_n,f_k)$ is equivalent to solve the equation
\begin{equation}\label{Eq:Arith} f_m+f_k=2f_n.\end{equation}
Equations of the type $Af_m+Bf_n+Cf_k=0$ have been studied by Schlickewei and Schmidt \cite{Schlickewei:1993}. Before stating our results we introduce some notions.

A linear recurrence $f_n$ of order $d$ is a complex sequence satisfying the recurrence
\[f_{n+d}=a_{d-1}f_{n+d-1}+\cdots+a_0f_n\]
with $a_i\in\C$ for $i=0,\ldots,d-1$, $a_0\not=0$ and the sequence does not satisfy such an equation with fewer summands. The companion polynomial $P$ is defined by
\[P(X)=X^d-a_{d-1}X^{d-1}-\cdots-a_0.\]
A linear recurrence is simple if its companion polynomial $P$ has
simple zeros only, and it is called non-degenerate if
$\alpha_i/\alpha_j$ is not a root of unity for any distinct zeros
$\alpha_i$ and $\alpha_j$ of $P$. Further, a recurrence is called
a unitary sequence if its companion polynomial possesses at least
one zero which is a root of unity.

Let $\alpha_1,\ldots,\alpha_r$ be the zeros of the companion
polynomial $P$ and assume that $\alpha_i$ is a zero of
multiplicity $\sigma_i$. Then we can write
\[f_n=\sum_{i=1}^r p_i(n)\alpha_i^n,\]
where $p_i(n)$ are polynomials of degree $<\sigma_i$.

Using the above cited result by Schlickewei and Schmidt we prove
the following

\begin{theorem}\label{Th:nondom}
Let $f_n$ be a non-degenerate and non-unitary recurrence with
companion polynomial $P$. Then there is a finite set
$S_0\subset\N^3$ such that all three-term arithmetic progressions
$(f_m,f_n,f_k)$ with $f_n\not=0$ satisfy $(m,n,k)\in S_0$
(isolated solutions) or one of the following three cases occurs:
\begin{itemize}
\item All but finitely many solutions to \eqref{Eq:Arith} are of
the form $m=k+a,n=k+b$, with $a,b\in\Z$ and
$P(X)|(X^a-2X^b+1)X^{-\min\{a,b,0\}}$. \item The recursion is of
the form
\begin{align}\label{Eq:THSC1}
\begin{split}
f_n=&\sum_{i=1}^r
c_i\left(\alpha_{2i-1}^n+\alpha_{2i}^n\frac{\alpha_{2i-1}^{a+c}+\alpha_{2i-1}^{b+c}}2 \zeta_i^c\right), \;\; \text{with}\\
0=&(\zeta_i^a+\zeta_i^b-4\zeta_i^c)+\zeta_i^a
\alpha_j^{b-a}+\zeta_i^b \alpha_j^{a-b}\,\,\text{or}\qquad
\end{split}\\\label{Eq:THSC2}
\begin{split}
f_n=&\sum_{i=1}^r
c_i\left(\alpha_{2i-1}^n+\alpha_{2i}^n(\alpha_{2i-1}^{a+c}+2\alpha_{2i-1}^{b+c}) \zeta_i^c\right), \;\; \text{with}\\
0=&(\zeta_i^a+4\zeta_i^b-\zeta_i^c)-2\zeta_i^a
\alpha_j^{b-a}-2\zeta_i^b \alpha_j^{a-b}\,\,\text{or}
\end{split}\\\label{Eq:THSC3}
\begin{split}
f_n=&\sum_{i=1}^r
c_i\left(\alpha_{2i-1}^n+\alpha_{2i}^n(2\alpha_{2i-1}^{a+c}+\alpha_{2i-1}^{b+c}) \zeta_i^c\right),\;\; \text{with}\\
0=&(4\zeta_i^a+\zeta_i^b-\zeta_i^c)-2\zeta_i^a
\alpha_j^{b-a}-2\zeta_i^b \alpha_j^{a-b},
\end{split}
\end{align}
where $j=2i-1,2i$, $c_i\in\C$, $\alpha_{2i-1}\alpha_{2i}=\zeta_i$ is an $M$-th root of unity with $M$
minimal for all $i=1,\ldots,r/2$.
Then according to \eqref{Eq:THSC1}, \eqref{Eq:THSC2} or
\eqref{Eq:THSC3} $(f_m,f_k,f_n)$ or $(f_k,f_n,f_m)$ or $(f_n,f_m,f_k)$ with $m=Mt+a,n=Mt+b,k=-Mt+c$ are
arithmetic progressions for all integers $t$.
\item The recursion is of the form
\[f_n=C(n-\gamma)2^{n/K}\zeta_K^n\]
where $\zeta_K$ is a $K$-th root of unity, with $\gamma,K\in\Z$
and $C\in\C$. Then $f_n,f_m$ and $f_k$ form an arithmetic
progression (arranged in some order) if $n=c2^s+\gamma,
m=c2^s+as+b, k=c2^s+a's+b'$ with $a,a',b,b',c$ integers for all
integers $s\geq 0$. Moreover $K$ and $c$ cannot be both positive.
\end{itemize}
\end{theorem}

We exclude the case $(f_m,0,f_k)$ since this leads to the equation $f_m=-f_k$ which is not
an essential restriction for so called symmetric recurrences. In order to
keep Theorem \ref{Th:nondom} as short as possible (which is not an easy task) we made this technical restriction.
Note that the other cases are essential restrictions for the recurrences.
Therefore excluding this case recurrences which admit infinitely
many three-term arithmetic progressions are in some way very special.
How special they are can be seen in Corollary \ref{Cor:Fib}.
However, it is no problem to include conditions under which $(f_m,0,f_k)$ is an
arithmetic progression.

Also, remark that we can bound the number of isolated solutions
$|S_0|$ but we cannot give an upper estimate for the ``maximum''
of these solutions. The reason lies in the use of the quantitative
version of the subspace theorem. We want to point out here that in
many important cases we can compute $S_0$ effectively. At least
this can be done for all binary and ternary recurrences (for
techniques to do see \cite{Mignotte:1975}).

If we restrict ourselves to recurrences defined over the integers, i.e. $f_n\in\Z$ for all $n\in \Z$, and consider only positive indices we obtain.

\begin{corollary}\label{Cor:Int}
Let $f_n$ be non-degenerate, non-unitary and be defined over the
integers. Moreover, assume $f_n$ provides infinitely many
three-term arithmetic progressions $(f_m,f_n,f_k)$ with $n,m,k>0$.
Then the companion polynomial $P(X)$ is one of the factors of
$\frac{X^a-2X^b+1}{X^d-1}$ with $a>b>0$ and $d=\gcd(a,b)$.
\end{corollary}

Note that the factorization of trinomials has been extensively
studied by Schinzel (see e.g. his book \cite{Schinzel:Pol}). In
particular the precise factorization of the polynomial
$X^a-2X^b+1$ for $a>b>0$ is known (see \cite{Schinzel:1962}).
Schinzel used the factorization of $X^a-2X^b+1$ to prove that
there exist no non-trivial four-term arithmetic progression in
sequences of the form $f_n=q^n$ and $q$ an irrational number (a
question due to Sierpi\'{n}ski \cite{Sierpinski:1954}). In Lemma
\ref{Lem:irr} we will give the factorization of $X^a+X^b-2$. These
results on the factorization of trinomials are crucial in the
proof of the following theorem on the binary recurrence case.

\begin{theorem}\label{Th:Bin}
Let $f_n$ be a non-degenerate and non-unitary binary recurrence,
which is defined over the rationals and contains infinitely many
three-term arithmetic progressions. Then $f_n$ fulfills one of the
following conditions
\begin{itemize}
\item The binary recurrence $f_n$ is of the form $f_n=R(n-\gamma)2^{\pm n}$, with $R\in\Q^*$ and $\gamma\in\Z$. Such recurrences admit arithmetic three term progressions $(f_m,f_n,f_k)$ with
\[m=\mp 2^{s\mp\gamma}\pm s, \quad n=\mp 2^{s\mp\gamma}\pm s\mp 1, \quad k=\mp 2^{s\mp\gamma}+\gamma\]
for all $s>\pm \gamma+\frac{1\pm 1}2$.
\item The sequence is listed in Table \ref{Tab:Sym} (up to a multiplication by a rational) and $f_m,f_n$ and $f_k$ form a three-term arithmetic progression (in some order) with $m=2t+a,n=2t+b$ and $k=-2t+c$ for all $t\in\Z$.
\item The companion polynomial of the recurrence $f_n$ is listed in Table \ref{Tab:Bin}.
\end{itemize}
\end{theorem}

\begin{table}[ht]
\caption{Sequences that contain infinitely many arithmetic
progressions involving $f_{2t+a},f_{2t+b}$ and $f_{-2t+c}$}
\label{Tab:Sym}
\begin{tabular}{|c|c|c|}
\hline
$f_n$ & $\alpha$ & $a,b,c$ \\\hline\hline
\multirow{4}*{\begin{tabular}{c} $f_n=C\left(\alpha^n+(-1)^{n+c}\alpha^{b+c-n}\frac{1+\alpha}2\right)$ \\
$C=1+\frac{(-1)^a\alpha^{-(a+c)}+(-1)^b\alpha^{-(b+c)}}2$ \end{tabular}} &
\multirow{2}*{$\alpha=2\pm \sqrt{5}$} & $a-b=1$ \\
& &  $b+c\equiv 0 \mod 2$ \\ \cline{2-3}
 & \multirow{2}*{$\alpha=-2\pm \sqrt{5}$} &  $a-b=1$ \\
& & $b+c\equiv 1 \mod 2$ \\ \hline
\multirow{4}*{\begin{tabular}{c} $f_n=C\left(\alpha^n+(-1)^{n+c}\alpha^{b+c-n}\frac{1+\alpha^3}2\right)$ \\
$C=1+\frac{(-1)^a\alpha^{-(a+c)}+(-1)^b\alpha^{-(b+c)}}2$ \end{tabular}} &
\multirow{2}*{$\alpha=\frac{1\pm\sqrt 5}2$} & $a-b=3$ \\
& &  $b+c\equiv 0 \mod 2$ \\\cline{2-3}
 & \multirow{2}*{$\alpha=\frac{-1\pm\sqrt 5}2$} & $a-b=3$ \\
& & $b+c\equiv 1 \mod 2$ \\ \hline
\multirow{4}*{\begin{tabular}{c} $f_n=C\left(\alpha^n+(-1)^{n+c+1}\alpha^{b+c-n}(2\alpha-1)\right)$ \\
$C=1+(-1)^b\alpha^{-(b+c)}-(-1)^a2\alpha^{-(a+c)}$\end{tabular}} &
\multirow{2}*{$\alpha=-1-\sqrt 2$} &  $a-b=1$ \\
& & $b+c\equiv 0 \mod 2$ \\ \cline{2-3}
 & \multirow{2}*{$\alpha=-\frac{1+\sqrt 5}2$} &  $a-b=1$ \\
& & $b+c\equiv 1 \mod 2$ \\ \hline
\multirow{4}*{\begin{tabular}{c} $f_n=C\left(\alpha^n+(-1)^{n+c}\alpha^{b+c-n}(\alpha-2)\right)$\\
$C=1+(-1)^a\alpha^{-(a+c)}-(-1)^b2\alpha^{-(b+c)}$ \end{tabular}} &
\multirow{2}*{$\alpha=-1-\sqrt 2$} &  $a-b=1$ \\
& & $b+c\equiv 1 \mod 2$ \\ \cline{2-3}
 & \multirow{2}*{$\alpha=-\frac{1+\sqrt 5}2$} &  $a-b=1$ \\
& & $b+c\equiv 0 \mod 2$ \\ \hline
\end{tabular}
\end{table}

\begin{table}[ht]
\caption{Companion polynomials of binary recursion containing
arithmetic progressions with $f_{n+a},f_{n+b}$ and $f_n$
involved.}\label{Tab:Bin}
\begin{tabular}{|c|c|c|}
\hline
$a$ & $b$ & $P(X)$ \\\hline\hline
\multirow{3}*{$3$} &\multirow{3}*{$1$} & $X^2+X-1$ \\ \cline{3-3}
 & & $X^2+X+2$ \\ \cline{3-3}
 & & $2X^2+2X+1$ \\ \hline
\multirow{3}*{$3$} & \multirow{3}*{$2$} & $X^2-X-1$ \\ \cline{3-3}
 & & $X^2+2X+2$ \\ \cline{3-3}
 & & $2X^2+X+1$  \\ \hline
\end{tabular}
\end{table}

Let us consider the Fibonacci sequence, i.e.
\[f_0=0, \quad f_1=1, \quad f_{n+2}=f_{n+1}+f_n, \;\;n\geq 0.\]
We obtain the following characterization of the Fibonacci sequence.

\begin{corollary}\label{Cor:Fib}
The only increasing, simple, non-degenerate and non-unitary
recursion $f_n$ defined over the rationals that contains
infinitely many three-term arithmetic progressions $(f_m,f_n,f_k)$
with $m,n,k\geq 0$, which additionally satisfies $f_0=0$ and
$f_1=1$ is the Fibonacci sequence.

Moreover, the Fibonacci sequence contains for $n\geq 0$ beside
the infinite family $(f_{n},f_{n+2},f_{n+3})$ of three-term arithmetic
progressions only the three-term arithmetic progressions
\begin{gather*}
(f_0=0,f_1=1,f_3=2), \;\; (f_0=0,f_2=1,f_3=2)\;\; \text{and}\;\;
(f_2=1,f_3=2,f_4=3).
\end{gather*}
The only four-term arithmetic progressions are
\[(f_0=0,f_1=1,f_3=2,f_4=3)\;\; \text{and}\;\; (f_0=0,f_2=1,f_3=2,f_4=3).\]
\end{corollary}

The condition non-unitary is essentially since the sequence
$f_n=\frac{2^n-(-1)^n}3$ fulfills the same properties as required
in the corollary and contain the infinite family of arithmetic
three-term progressions $(f_{2t},f_{2t-1},f_1=f_2)$. However, by
simple growth estimates we can show that this sequence is the only
exception.

\begin{corollary}\label{Cor:unitary}
Omitting the condition non-unitary in Corollary \ref{Cor:Fib}, we
have $f_n=\frac{2^n-(-1)^n}3$ or $f_n$ is the Fibonacci sequence.
\end{corollary}

Although Theorem \ref{Th:Bin} is long and technical the case of
ternary sequences is much easier, since the so-called symmetric
and exceptional cases do not occur. Therefore we show

\begin{theorem}\label{Th:Ternary}
Let $f_n$ be a non-degenerate, non-unitary, ternary recurrence, which is defined over the rationals and contains infinitely many arithmetic progressions. Then $f_n$ has companion polynomial listed in Table \ref{Tab:Ter}.
\end{theorem}

\begin{table}[ht]
\caption{Companion polynomials of ternary recursion containing
arithmetic progressions with $f_{n+a},f_{n+b}$ and $f_n$
involved.}\label{Tab:Ter}
\begin{tabular}{|c|c|c|}
\hline
$a$ & $b$ & $P(X)$ \\\hline\hline
\multirow{3}*{$4$} & \multirow{3}*{$1$} & $X^3+X^2+X-1$ \\ \cline{3-3}
 & & $X^3+X^2+X+2$ \\ \cline{3-3}
 & & $2X^3+2X^2+2X+1$ \\ \hline
\multirow{3}*{$4$} & \multirow{3}*{$3$} & $X^3-X^2-X-1$ \\ \cline{3-3}
 & & $X^3+2X^2+2X+2$ \\ \cline{3-3}
 & & $2X^3+X^2+X+1$  \\ \hline
\multirow{2}*{$7$} & \multirow{2}*{$2$} & $X^3+X^2+1$ \\ \cline{3-3}
& & $X^3-X-1$\\ \hline
\multirow{2}*{$7$} & \multirow{2}*{$5$} & $X^3+X^2-1$\\ \cline{3-3}
& & $X^3+X-1$ \\ \hline
\end{tabular}
\end{table}

\section{Notation and Linear equations in recurrences}

We start this section with some notions. In the sequel we assume
that $f_n$ is a non-degenerate and non-unitary linear recurrence
sequence with companion polynomial~$P$. Let
$\alpha_1,\ldots,\alpha_r$ be the zeros of $P$. We call $f_n$
symmetric if $r$ is even and the zeros $\alpha_1,\ldots,\alpha_r$
can be arranged such that $(\alpha_i\alpha_{i+1})^M=1$ for each
odd $1\leq i<r$. We call $f_n$ exceptional if there exists an
integer $N>0$ such that each $\alpha_i$ is a rational power of
$N$, each $|\alpha_i|>1$ or each  $|\alpha_i|<1$ and
$p_i(n)=\gamma_i(n-\gamma)$ with $\gamma\in\Q$. Note that a
recurrence cannot be both symmetric and exceptional. We are
interested in the two equations
\begin{equation}\label{Eq:twoterm}
Af_n=Bf_m
\end{equation}
and
\begin{equation}\label{Eq:threeterm}
Af_n+Bf_m+Cf_k=0, \quad f_nf_mf_k\not=0
\end{equation}
where $ABC\not=0$. These equations were investigated by Laurent \cite{Laurent:1987} and Schlickewei and Schmidt \cite{Schlickewei:1993}, respectively. The next three Propositions are
reformulations of \cite[Proposition 1 and 2, Theorem 1 and 2]{Schlickewei:1993}.

Let us consider the case where $f_n$ is neither symmetric nor exceptional. Then we have

\begin{proposition}\label{Prop:nor}
Let $A,B,C$ be non-zero constants and let $f_n$ be neither symmetric nor exceptional. Then all solutions to \eqref{Eq:twoterm} but finitely many are contained in the one parameter family $n=t+a$ and $m=t+b$ for certain $a,b\in \Z$. Moreover all but finitely many solutions satisfy
\begin{equation}\label{Eq:Nor2}
Ap_i(n)\alpha_i^n=Bp_i(m)\alpha_i^m.
\end{equation}
All solutions to the ternary equation \eqref{Eq:threeterm} but finitely many are contained in one of finitely many families of the form
\begin{equation}\label{Fam:Nor3}
\mathcal F_j :\quad n=k+a_j, \;\; m=k+b_j, \quad a_j,b_j\in \Z
\end{equation}
and satisfy the polynomial identity
\begin{equation}\label{Eq:Nor3}
Ap_i(n)\alpha_i^n+Bp_i(m)\alpha_i^m+Cp_i(k)\alpha_i^k=0.
\end{equation}
\end{proposition}

Now we consider the symmetric case. Assume we have arranged the roots as described above. In this case further solutions may occur:

\begin{proposition}\label{Prop:sym}
Let $A,B,C$ be non-zero constants and let $f_n$ be symmetric. Then the equation \eqref{Eq:twoterm} has the additional family of solutions $n=Mt+a'$ and $m=-Mt+b'$ for certain $a',b'\in\Z$. These solutions satisfy the system
\begin{equation}\label{Eq:Sym2}
\begin{split}
Ap_i(n)\alpha_i^n=&Bp_{i+1}(m)\alpha_{i+1}^m, \\
Ap_{i+1}(n)\alpha_{i+1}^n=&Bp_i(m)\alpha_i^m,
\end{split}
\end{equation}
for all odd $i$ with $1\leq i\leq r$.
Solutions to the ternary equation \eqref{Eq:threeterm} may lie in one of the additional families of solutions
$\mathcal S^{(n)}_j,\mathcal S^{(m)}_j$ or $\mathcal S^{(k)}_j$, where e.g.
\begin{equation}\label{Fam:Sym3}
\mathcal S^{(k)}_j: \quad n=Mt+a^{(k)}_j, \;\; m=Mt+b^{(k)}_j \;\; k=-Mt+c^{(k)}_j;
\end{equation}
where the $a$'s, $b$'s and $c$'s are integers. All additional solutions satisfy a corresponding
system of equations, e.g. for the family $\mathcal S^{(k)}_j$ we have
\begin{equation}\label{Eq:Sym3}
\begin{split}
Ap_i(n)\alpha_i^n+Bp_i(m)\alpha_i^m+Cp_{i+1}(k)\alpha_{i+1}^k&=0,\\
Ap_{i+1}(n)\alpha_{i+1}^n+Bp_{i+1}(m)\alpha_{i+1}^m+Cp_i(k)\alpha_i^k&=0,
\end{split}
\end{equation}
for all odd $i$ with $1\leq i\leq r$. The other equations are
obtained by permuting indices.
\end{proposition}

Finally, in the exceptional case we obtain

\begin{proposition}\label{Prop:exc}
Equation \eqref{Eq:twoterm} has no additional solutions in view of Proposition \ref{Prop:nor}. All but finitely many solutions to the ternary equation \eqref{Eq:threeterm} satisfy \eqref{Eq:Nor3}. But additional solutions may
lie in one of the finitely many exceptional families $\mathcal E^{(n)}_j,\mathcal E^{(m)}_j$ or $\mathcal E^{(k)}_j$, where e.g.
\begin{equation}\label{Fam:Exc3}
\mathcal E^{(n)}_j: \quad n=c_jN^s+\gamma, \;\; m=c_jN^s+as+b_j
\;\; k=c_jN^s+a's+b'_j.
\end{equation}
These additional solutions appear only if all $p_i(n)=\gamma_i(n-\gamma)$. Further $c_j\in\Q^*$ and $a,a',b_j,b_j'\in\Q$ are such that $(n(s),m(s),k(s))\in\Z^3$ for each $s\in\Z$, $s\geq 0$.
\end{proposition}

\section{Proof of Theorem \ref{Th:nondom}}

For the proof of Theorem \ref{Th:nondom} we have to consider
equation \eqref{Eq:threeterm} with one of $A,B,C$ is equal to $-2$
and the other coefficients are equal to $1$. The case where
$f_mf_nf_k=0$ has to be considered separately. We divide the proof
of Theorem \ref{Th:nondom} into the obvious three subcases, i.e.
$f_n$ is symmetric, exceptional or neither of them. Let us start
with the case where $f_n$ is neither symmetric nor exceptional.

\subsection{The general case}\label{sub:gen}
First we assume $f_mf_nf_k\not=0$. Then we may assume that all but finitely many solutions are of the form
$m=k+a,n=k+b$ with $a>b>0$ and they satisfy the equation
\[Ap_i(k+a)\alpha_i^a+Bp_i(k+b)\alpha_i^b+Cp_i(k)=0\]
for all $1\leq i\leq r$ (see Proposition \ref{Prop:nor}). Fix the index $i$ and write for simplicity $\alpha=\alpha_i$ and $p_i(k)=p(k)=A_d k^d+\cdots +A_0$. Considering the equation above for $k\rightarrow \infty$ we see that the equation has to be satisfied polynomial. Assume $d>0$. We compair the coefficients of $k^d$ in the equation and find
\begin{equation}\label{Eq:GCd}
AA_d\alpha^a+BA_d\alpha^b+CA_d=0
\end{equation}
and for $k^{d-1}$ we find
\begin{equation}\label{Eq:GCd-1}
A\alpha^a(A_d d a+A_{d-1})+B(A_d d b+A_{d-1})+CA_{d-1}=0
\end{equation}
Subtracting equation \eqref{Eq:GCd} from \eqref{Eq:GCd-1} and after some calculations we obtain the system
\begin{equation*}
\begin{split}
 A \alpha^a + B \alpha^b & =-C\\
 A a \alpha^a +B b\alpha^b & =0
\end{split}
\end{equation*}
Solving for $\alpha^a$ and $\alpha^b$ yields $\alpha^a=-bBC$ and $\alpha^b=aAC$. Assume $A=-2$ then we
have $\alpha^a=-b$ and $\alpha^b=-2a$. Taking the first relation to the $b$-th power and inserting the second we obtain $(-b)^b=(-2a)^a$. In the case of $B=-2$ or $C=-2$ we obtain $(2b)^b=a^a$ or $(2b)^b=(-2a)^a$, respectively. The last equation has obviously no integral solution with $a>b>0$. The other two cases have also no solution because of the next lemma.

\begin{lemma}
The equation $a^a=(2b)^b$ has no positive integral solution.
\end{lemma}

\begin{proof}
First, note that the equation implies $a>b$, i.e. $a=xb$ with $x>1$, $x=p/q$ and $p,q\in\Z$ with $\gcd(p,q)=1$.
Inserting for $a=xb$ the equation is equivalent to
\[b^{x-1}x^x=2\]
taking $q$-th powers we have
\[b^{p-q}x^p=2^q\]
a rational equation. Let $r\not=2$ be a prime dividing $p$. Computing the $r$-adic valuations on the left and right hand side we obtain
\[\beta_r(p-q)+p\alpha_r=0,\]
where $\beta_r$ and $\alpha_r$ are the $r$-adic valuations of $b$ and $p$. Since $\alpha_r>0$, $\beta_r\geq 0$ and $p>q$ we have a contradiction. Therefore $p=2^k$ and we consider $2$-adic valuations:
\[\beta_2(p-q)+pk=q.\]
Since $p>q$ we obtain again a contradiction unless $k=0$. Hence $x=1/q\leq 1$ again a contradiction, i.e. the equation has no solution.
\end{proof}

Therefore all $p_i$ are constant and all $\alpha$ have to satisfy either of the equations
\[-2X^a+X^b+1, \quad X^a-2X^b+1, \quad X^a+X^b-2,\]
with $a>b>0$.

Now we consider the case $f_mf_nf_k=0$. Since we excluded the case $(f_m,0,f_k)$ we are lead to the equation $2f_n=f_k$. By Proposition \ref{Prop:nor} we have
\begin{equation}\label{Eq:norzero}
2p_i(k+a)\alpha_i^{k+a}=p_i(k)\alpha_i^k.
\end{equation}
Dividing through $\alpha_i^k$ and then taking the limit $k\rightarrow \infty$ we obtain
$\alpha_i^a=1/2$ for all $1\leq i \leq r$. Since the recurrence $f_n$ is non-degenerate we find
$r=1$ and $\alpha_1=\alpha=2^{-1/a}\zeta_a$ where $\zeta_a$ is some $a$-th root of unity. If we insert this into
\eqref{Eq:norzero} we obtain $p(k+a)=p(k)$ which on the other hand tells us $p(k)$ is constant. Therefore
$f_n=c2^{-n/a}$. On the other hand we have $f_m=0$ hence $c=0$. Therefore $f_n$ is a constant recurrence which we excluded.

\subsection{The symmetric case}
Now, let us treat the symmetric case. Let us write $\alpha_i\alpha_{i+1}=\zeta_i$ with $\zeta_i^M=1$ and as the first case let us assume $f_mf_nf_k\not=0$. In this case all solutions but finitely many lie in one of the families $\mathcal F_j$ or in $\mathcal S^{(m)}_j, \mathcal S^{(n)}_j$ or $\mathcal S^{(k)}_j$. The case where the solution lies in $F_j$ is identical with the case treated in the subsection above. Therefore we may assume $m=Mt+a, n=Mt+b$ and $k=-Mt+c$. According to
Proposition \ref{Prop:sym} we have to distinguish three cases. For each odd $i$ we have
\begin{equation}\label{Eq:symS1}
\begin{split}
Ap_i(Mt+a)\alpha_i^{Mt+a}+Bp_i(Mt+b)\alpha_i^{Mt+b}+Cp_{i+1}(-Mt+c)\alpha_{i+1}^{-Mt+c}&=0;\\
Ap_{i+1}(Mt+a)\alpha_{i+1}^{Mt+a}+Bp_{i+1}(Mt+b)\alpha_{i+1}^{Mt+b}+Cp_{i}(-Mt+c)\alpha_{i}^{-Mt+c}&=0
\end{split}
\end{equation}
or an equation which is obtained from \eqref{Eq:symS1} by permuting $A,B$ and $C$.

Let us multiply the first equation by $\alpha_i^{-Mt}$ and the second by $\alpha_{i+1}^{-Mt}$. Then we get
\begin{equation*}
\begin{split}
Ap_i(Mt+a)\alpha_i^a+Bp_i(Mt+b)\alpha_i^b+Cp_{i+1}(-Mt+c)\alpha_{i+1}^c&=0;\\
Ap_{i+1}(Mt+a)\alpha_{i+1}^a+Bp_{i+1}(Mt+b)\alpha_{i+1}^b+Cp_{i}(-Mt+c)\alpha_{i}^c&=0.
\end{split}
\end{equation*}
We note that the polynomials $p_i$ and $p_{i+1}$ are of the same degree, because otherwise either the first equation
or the second would yield a contradiction if we divide through $p_i$ and $t$ tends to infinity.
Now let us assume that $p_i$ and $p_{i+1}$ are of degree $d\geq 1$ and assume
\begin{gather*}
p_i(t)=A_d^{(i)}t^d+A_{d-1}^{(i)}t^{d-1}+\cdots;\\
p_{i+1}(t)=A_d^{(i+1)}t^d+A_{d-1}^{(i+1)}t^{d-1}+;\cdots
\end{gather*}
and write $q_i=A_d^{(i+1)}/A_d^{(i)}$. Comparing coefficients of
$t^d$ and $t^{d-1}$ yields the system of equations
\begin{gather*}
AA_d^{(i)}M^d\alpha_i^a+BA_d^{(i)}M^d\alpha_i^b+CA_d^{(i+1)}M^d\alpha_{i+1}^c=0;\\
AA_d^{(i+1)}M^d\alpha_{i+1}^a+BA_d^{(i+1)}M^d\alpha_{i+1}^b+CA_d^{(i)}M^d\alpha_{i}^c=0;\\
\begin{split}
M^{d-1}A\alpha_i^a(A_d^{(i)}ad+A_{d-1}^{(i)})+M^{d-1}B\alpha_i^b(A_d^{(i)}bd+A_{d-1}^{(i)})\\
+M^{d-1}C\alpha_{i+1}^c(A_d^{(i+1)}cd+A_{d-1}^{(i+1)})&=0;
\end{split}\\
\begin{split}
M^{d-1}A\alpha_{i+1}^a(A_d^{(i+1)}ad+A_{d-1}^{(i+1)})+M^{d-1}B\alpha_{i+1}^b(A_d^{(i+1)}bd+A_{d-1}^{(i+1)})\\
+M^{d-1}C\alpha_{i}^c(A_d^{(i)}cd+A_{d-1}^{(i)})=0;
\end{split}
\end{gather*}
and by straightforward calculations we obtain
\begin{equation}\label{Eq:symd>0}
\begin{split}
A\alpha_i^a+B\alpha_i^b+Cq_i\alpha_{i+1}^c&=0;\\
A\alpha_{i+1}^a+B\alpha_{i+1}^b+\frac C{q_i}\alpha_i^c&=0;\\
Aa\alpha_i^a+Bb\alpha_i^b+Ccq_i\alpha_{i+1}^c&=0;\\
Aa\alpha_{i+1}^a+Bb\alpha_{i+1}^b+\frac C{q_i}c\alpha_i^c&=0.
\end{split}
\end{equation}
Computing from the first equation $Cq_i\alpha_{i+1}^c$ and inserting into the third equation we get
\[A(a-c)\alpha_i^a+B(b-c)\alpha_i^b=0.\]
The second and forth equation lead to the same relation for $\alpha_{i+1}$. Therefore either $\alpha_i=\alpha_{i+1}$
or the sequence is degenerate or $a=c$ and $b=c$, hence in any case we obtain a contradiction. Therefore we have $d=0$.

Let us investigate the first two equations of \eqref{Eq:symd>0}. By using the fact that $\alpha_{i+1}^n=\alpha_i^{-n}\zeta_i^n$ for all integers $n$ the second equation can be rewritten as
\[A\alpha_{i}^{-a}\zeta_i^a+B\alpha_{i}^{-b}\zeta_i^b=-\frac {C\zeta_i^c}{q_i\alpha_{i+1}^c}.\]
Since the right side is obviously not zero we also have $A\alpha_{i}^{-a}\zeta_i^a+B\alpha_{i}^{-b}\zeta_i^b\not=0$ and therefore we can write
\[-\frac{C\zeta_i^c}{A\alpha_{i}^{-a}\zeta_i^a+B\alpha_{i}^{-b}\zeta_i^b}=q_i\alpha_{i+1}^c.\]
Inserting into the first equation yields
\[A\alpha_i^a+B\alpha_i^b-\frac{C^2\zeta_i^c}{A\alpha_{i}^{-a}\zeta_i^a+B\alpha_{i}^{-b}\zeta_i^b}=0\]
or in expanded form
\begin{equation*}
(A^2\zeta_i^a+B^2\zeta_i^b-C^2\zeta_i^c)+AB\zeta_i^a\alpha_i^{b-a}+AB\zeta_i^b\alpha_i^{a-b}=0.
\end{equation*}
Note that $\alpha_{i+1}$ satisfies the same equation. Let us assume $a>b$, then $\alpha_i$ and $\alpha_{i+1}$
are of the form $\xi_i\gamma_i^{1/(a-b)}$ or $\xi_{i+1}\gamma_{i+1}^{1/(a-b)}$, where $\xi_i,\xi_{i+1}$ are roots of unity and $\gamma_i$ and $\gamma_{i+1}$ are roots of the polynomial
\begin{equation*}
X^2+\frac{A^2\zeta_i^a+B^2\zeta_i^b-C^2\zeta_i^c}{AB\zeta_i^b} X+\zeta_i^{a-b}.
\end{equation*}
Note that $\alpha_i$ and $\alpha_{i+1}$ cannot be both of the form
$\xi\gamma_i^{1/(a-b)}$ or $\xi\gamma_{i+1}^{1/(a-b)}$, with $\xi$
some root of unity, since otherwise the recurrence would be
degenerate. Inserting for $A,B$ and $C$ the values $1$ and $-2$
according to the cases that may occur we obtain the sequences
listed in Table \ref{Tab:Sym}.

Before we proceed with the case $f_nf_mf_k=0$ we want to demonstrate this case by an example: Let us choose $r=2$ and $\alpha_1=2+\sqrt{5}$, i.e. $\alpha_2=2-\sqrt{5}$. Moreover, we choose $a=2$ and $b=c=1$. Therefore all sequences of the form
\[f_n=c_0\left((2+\sqrt 5)^n-(2-\sqrt 5)^n\frac {47+21\sqrt 5}2 \right)\]
have infinitely many three-term arithmetic progressions. We choose $c_0=\frac{1+\frac{21\sqrt 5-47}2}{15}$ and insert for $n=0$ and $n=1$ and observe that $2+\sqrt 5$ is a root of $X^2-4X-1$. Then we see that the sequence $f_n$ comes from the recurrence
\[f_{n+2}=4f_{n+1}+f_n, \quad f_0=-3, \quad f_1=2.\]
Therefore $f_n$ is defined over the integers and contains
infinitely many three-term arithmetic progressions
$(f_{2n+2},f_{-2n+1},f_{2n+1})$.

Now let us consider the case $f_mf_nf_k=0$. As arguing in the previous case we have to deal with the equation $f_n=2f_k$. By Proposition \ref{Prop:sym} we either have
\[2p_i(n)\alpha_i^n=p_i(k)\alpha_i^k\]
for each $i$ or
\[2p_i(n)\alpha_i^n=p_{i+1}(k)\alpha_{i+1}^k;\qquad 2p_{i+1}(n)\alpha_{i+1}^n=p_i(k)\alpha_i^k\]
for each odd $i$. The first equation corresponds to the case
treated above. Therefore, we may assume $n=Mt+a$ and $k=-Mt+b$. We
multiply the first equation by $\alpha_{i+1}^{Mt}/p_i(Mt+a)$ and
the second by  $\alpha_i^{Mt}/p_{i+1}(Mt+a)$ and take the limit
for $t\rightarrow \infty$. Then we get
\begin{equation}\label{Eq:sym0case}
2\alpha_i^a=\alpha_{i+1}^b q_i, \qquad
2\alpha_{i+1}^a=\alpha_{i+1}^b 1/q_i,
\end{equation}
where $q_i=\lim_{t\rightarrow \infty} p_i(tM+a)/p_{i+1}(-Mt+b)$.
Now eliminating $q_i$ from the first equation of \eqref{Eq:sym0case} yields
\[4\alpha_i^{a-b}=\alpha_{i+1}^{b-a}\]
and multiplying by $\alpha_{i+1}^{a-b}$ yields
\[4\zeta^{a-b}=1\]
a contradiction.

Note that in the case of $f_m=-f_k$ we obtain by the same computation $\zeta^{a-b}=1$ which would yield further solutions. Therefore we have excluded these cases in Theorem \ref{Th:nondom}.

\subsection{The exceptional case}

First, we consider the case $f_mf_nf_k\not=0$. Then by Proposition \ref{Prop:exc} we may assume
$n=cN^s+\gamma, m=cN^s+as+b$ and $k=cN^s+a's+b'$, $p_i(n)=\gamma_i(n-\gamma)$, $N=\alpha_i^{q_i}$ for some rational number $q_i$ and \eqref{Eq:Nor3} is satisfied. In order to treat several cases at once we assume
\begin{equation*}
n=cN^s+as+b,\quad  m=cN^s+a's+b',\quad k=cN^s+a''s+b''
\end{equation*}
with $a\geq a' \geq a''$ and one of $a,a'$ and $a''$ is zero and the corresponding $b$ is equal to $\gamma$.
Then by Proposition \ref{Prop:exc} we know that for all $\alpha_i$, $i=1,\ldots,r$ we have
\begin{multline*}
A\gamma_i (c N^s+as+b-\gamma)\alpha_i^{cN^s+as+b}+B\gamma_i(c N^s+a's+b'-\gamma)\alpha_i^{cN^s+a's+b'}\\
+C\gamma_i(c N^s+a''s+b''-\gamma)\alpha_i^{cN^s+a's+b'}=0.
\end{multline*}
For reasons of notation let us drop the indices. Then the equation
above can be written as
\begin{equation}\label{Eq:excMain}
\begin{split}
&Ac\alpha^{(q+a)s+b}+Bc\alpha^{(q+a')s+b'}+Cc\alpha^{(q+a'')s+b''}\\
+&A(as+b-\gamma)\alpha^{as+b}+B(a's+b'-\gamma)\alpha^{a's+b'}\\
+&C(a''s+b''-\gamma)\alpha^{a''s+b''}=0.
\end{split}
\end{equation}
Let us assume we have $|\alpha|>1$ (for all $i$). Therefore we have $q>0$. So the maximal coefficient of $s$ in the exponents of \eqref{Eq:excMain} is $q+a$. By dividing by $\alpha^{(q+a)s}$ we see that $(q+a)s+b$ cannot be the only maximal exponent. Otherwise every other term than $Ac\alpha^b$ would converge to $0$ and hence $Ac=0$, a contradiction. So either $q+a=a$ or $q+a=q+a'$. The first case can be excluded since otherwise $q=0$ and hence $N=1$. Therefore we have $a=a'$. If a third exponent would be also maximal we would have again either $q=0$ or $a=a'=a''=0$. Now the second case would yield a situation as treated in subsection \ref{sub:gen}. Since the leading terms must cancel, we get
\[Ac\alpha^{(q+a)s+b}+Bc\alpha^{(q+a)s+b'}=0\]
and in particular
\[-\frac AB=\alpha^{b'-b}.\]
Since $\alpha$ is not a root of unity we must have $AB=-2$, hence
$C=1$. Moreover, since $m$ and $n$ must be both integers also
$b-b'$ is an integer, hence $\alpha^K=2$, with $K=|b-b'|\in \Z$
(remember $|\alpha|>1$). Of course $b$ and $b'$ depend on the
exceptional family e.g. $\mathcal E_j^{(k)}$ but not on the root
$\alpha_i$. So for all roots $\alpha_i$ we have the same $K$,
hence there exists only one $\alpha$ since otherwise our
recurrence would be degenerate. Moreover, observe that also
$A(b-b')>0$ holds. Using these facts we have
\[B\alpha^{a's+b'}=B\alpha^{as+b+(b'-b)}=-A\alpha^{as+b}\]
and together with \eqref{Eq:excMain} we get
\[c\alpha^{(q+a'')s+b''}+A(b-b')\alpha^{as+b}+C(a''s+b''-\gamma)\alpha^{a''s+b''}=0.\]
Now the highest exponent is either $as+b$ or $(q+a'')s+b''$, but in any case $a''s+b''$ is smaller (otherwise $q=0$).
Since a single maximum yields a contradiction we deduce similar as above
\[c\alpha^{b''}+A(b-b')\alpha^{b}=0,\]
which also implies
\[C(a''s+b''-\gamma)\alpha^{a''s+b''}=0,\]
i.e. $a''=0$ and $b''=\gamma$ and so the equation above turns into
\begin{equation}\label{Eq:excC1} c+A(b-b')\alpha^{b-\gamma}=0.\end{equation}
From the equation above we also deduce that $b-\gamma\in\Z$ since otherwise $\alpha^{b-\gamma}$ is irrational and so also the left side of the equation, a contradiction. Note that in the case $|\alpha|<1$ and by assuming $a''\geq a' \geq a$ we obtain the same conclusions, except $K=-|b-b'|$ and $A(b-b')<0$. Let us now assume $n,m,k>0$ then we have $c>0$ and additionally let us assume $|\alpha|>1$ then we also have $K\geq 1\in \Z$ and by \eqref{Eq:excC1} we deduce $\alpha\not \in \R^+$ but $\alpha^{b-\gamma}\in\Q^-$ because of \eqref{Eq:excC1}. But the rational power of a negative rational never can be $2$, hence a contradiction. This shows that $c$ and $K$ cannot be both positive.

Now we want to prove that $a,a',b,b'$ and $c$ are integral. From the paragraph above we may assume $N=2$.
Since for all $s\in\Z$ with $s>0$ the quantities
\[n(s)=c2^s+\gamma,\quad m(s)=c2^s+as+b, \quad k(s)=c2^s+as+b'\]
must be integers and since $2^s$ and $as$ are periodic modulo each prime $p>2$ with period dividing $p-1$ and $p$ respectively and since $2^s$ is not constant modulo $p$ we deduce that the denominator of $c$ is a power of $2$ and therefore $c2^s$ is for large $s$ an integer, which yields that $a,a',b,b'$ and $\gamma$ are integers (at least for large $s$ and hence for all $s$). Therefore also $c$ has to be an integer.

Now let us consider the case, where $f_n$, $f_m$ or $f_k$ vanishes. This leads to an equation of the
form \eqref{Eq:twoterm}, but Proposition \ref{Prop:exc} tells us that such an equation has no additional solutions.

\subsection{Proof of Corollary \ref{Cor:Int}}

The last subsection of this section is devoted to the proof of
Corollary \ref{Cor:Int}. First, we note that since we allow only
positive indices the symmetric case is excluded. The exceptional
case is also excluded since $\alpha$ with $\alpha^K=2$ has to be
an (algebraic integer), hence $|\alpha|>1$. But in this case we
have $K\geq 1$, hence by Theorem 1 we have $c<0$, contradicting
the fact $n,k,m>0$. So only the general case remains. But Lemma
\ref{Lem:irr} below will show that $2X^a-X^b-1=(X^d-1)g(X)$ with
$g(X)$ irreducible and $d=\gcd(a,b)$. Hence $P(X)=g(X)$ but has no
integral roots, hence $f_n\not\in\Z$ if $n$ is large. A simlar
argument applies to the case $X^a+X^b-2$, but here we conclude
$f_n\not\in\Z$ if $n$ tends to $-\infty$.

\section{The binary case}

\subsection{Exceptional case}
Since $f_n$ is defined over the rationals and $\alpha=2^K$ we have $K=\pm 1$, i.e. $f_n=R(n-\gamma)2^{Kn}$, where $R\in \Q^*$. Therefore we have $q=a=a'=K$ and without loss of generality we may assume $A=1$. Then we have $b-b'=K$ and $c+A(b-b')\alpha^{b-\gamma}=c+K2^{K(b-\gamma)}=0$. Moreover we have
\begin{align*}
m&=c2^s+as+b=-K2^{s+Kb-K\gamma}+Ks+b;\\
n&=c2^s+a's+b'=-K2^{s+Kb-K\gamma}+Ks+b-K;\\
k&=c2^s+\gamma=-K2^{s+Kb-K\gamma}+\gamma.
\end{align*}
In particular substituting $s$ for $s+Kb$ we see that $(f_m,f_n,f_k)$ is an three-term arithmetic progression if
\[m=-K2^{s-K\gamma}+Ks, \quad n=-K2^{s-K\gamma}+Ks-K, \quad k=-K2^{s-K\gamma}+\gamma.\]
For $s>K\gamma+\frac{1+K}2$ these are distinct integers. Substituting $1$ and $-1$ for $K$ we get the statement for the exceptional case.

\subsection{Symmetric case}
Let us now consider the symmetric case. We keep the notation of
the previous section. Since $f_n=c_1\alpha_1^n+c_2\alpha_2^n$ is
defined over the rationals we have $\alpha_1\alpha_2=\pm 1$. In
the case of $\alpha_1\alpha_2=1$ the first two equations of
\eqref{Eq:symd>0} yield $\alpha_1=\alpha_2=q=1$, which is
excluded. Therefore we have $\alpha_1\alpha_2=-1$.  The case
$C=-2$ yields polynomials of the form
\[X^2+4X-1, \quad X^2-4X-1,\quad X^4+6X^2+1,\quad X^4-2X^2+1\]
and in the case of $B=-2$ or $A=-2$ we obtain the polynomials
\[X^2-2X+1, \quad X^2+X-1, \quad X^2-X-1, \quad X^2+2X-1,\quad X^2-2X-1,\quad X^4-3X^2+1,\]
where $X$ is of the form $x^{a-b}$ or $x^{b-a}$ depending on the sign of $a-b$. Since $\alpha_1,\alpha_2$ have to be quadratic integers not roots of unity, the only possibilities are in the case of $C=-2$
\[\alpha_1=\pm 2\pm \sqrt{5}\;\; \text{and} \;\; \alpha_1=\frac{\pm 1\pm \sqrt{5}}2\]
and in the case of $B=-2$ or $A=-2$
\[\alpha_1=\pm 1\pm \sqrt{2}\;\; \text{and} \;\; \alpha_1=\frac{\pm 1\pm \sqrt{5}}2.\]
Note that except $\pm 2 \pm \sqrt{5}$ all of these are fundamental units and we have $\pm 2 \pm \sqrt{5}=\left(\frac{\pm 1\pm \sqrt{5}}2\right)^3$ choosing the signs aproperiately. In particular for all these integers we have to choose $a$ and $b$ such that $|a-b|=1$ or $|a-b|=3$. The last case may only occur for $\alpha_1=\frac{\pm 1\pm \sqrt{5}}2$ and $C=-2$.

Let us consider the case $C=-2$ in more detail. In this case we may assume without loss of generality $a>b$. We know that
\[q=\frac{\alpha_1^{a+c}+\alpha_1^{b+c}}2 (-1)^c\]
and therefore
\[f_0=C_0\left(1+\frac{\alpha_1^{a+c}+\alpha_1^{b+c}}2 (-1)^c\right)\]
and
\[f_1=C_0\left(\alpha_1+\frac{\alpha_1^{a+c-1}+\alpha_1^{b+c-1}}2 (-1)^{c+1}\right).\]
If $f_n$ is defined over the rationals then for a $C_0$ such that $f_0$ is rational also $f_1$ has to be rational.
If we choose $C_0=1+\frac{\alpha_2^{a+c}+\alpha_2^{b+c}}2 (-1)^c$ we certainly have $f_0\in\Q$ since this is the norm of $C_0$. Hence he have to consider $f_1$:
\begin{align*}
f_1=& \left(1+\frac{\alpha_2^{a+c}+\alpha_2^{b+c}}2 (-1)^c\right)\left(\alpha_1+\frac{\alpha_1^{a+c-1}+\alpha_1^{b+c-1}}2 (-1)^{c+1}\right)\\
=&\alpha_1+\stackrel{R}{\overbrace{\frac{\alpha_2^{a+c-1}+\alpha_2^{b+c-1}}2(-1)^{c+1}+\frac{\alpha_1^{a+c-1}+\alpha_1^{b+c-1}}2 (-1)^{c+1}}}\\
&-\stackrel{Q}{\overbrace{\frac{(\alpha_2^{a+c}+\alpha_2^{b+c})(\alpha_1^{a+c-1}+\alpha_1^{b+c-1})}4}}.
\end{align*}
We note that $R$ is a rational number since it is the trace of an algebraic number. The numerator of $Q$ turns into
\begin{multline*}
\alpha_2((-1)^{a+c-1}+(-1)^{b+c-1})+(-1)^{b+c}\alpha_1^{a-b-1}+(-1)^{b+c-1}\alpha_2^{a-b+1}=\\
(-1)^{b+c}(\alpha_1^{a-b-1}-\alpha_2^{a-b+1}).
\end{multline*}
In the case of $a-b=1$ we obtain
\[\alpha_1-\frac{(-1)^{b+c}(1-\alpha_2^2)}4\]
is rational. If we try all possibilities we see that this is possible if and only if
$\alpha_1=2\pm\sqrt{5}$ and $b+c\equiv 0 \mod 2$ or $\alpha_1=-2\pm\sqrt{5}$ and $b+c\equiv 1 \mod 2$.
In the case of $a-b=3$ we deduce that
\[\alpha_1-\frac{(-1)^{b+c}(\alpha_1^2-\alpha_2^4)}4\]
is rational. Note $a-b=3$ is only possible if $\alpha_1=\frac{\pm
1\pm\sqrt 5}2$. Therefore we see that
$\alpha_1=\frac{1\pm\sqrt{5}}2$ if $b+c\equiv 0 \mod 2$ and
$\alpha_1=\frac{-1\pm\sqrt{5}}2$ otherwise.

Now, let us consider the case $A=-2$. We may assume $a-b=1$ (note
$b-a=1$ yields the same computations and the same results as case
$B=-2$ and $a-b=1$ treated below) we have
\[q=(-1)^c(\alpha_1^{b+c}-2\alpha_1^{a+c})\]
and therefore we have
\[f_0=C_0\left(1+(-1)^c(\alpha_1^{b+c}-2\alpha_1^{a+c})\right)\]
and
\[f_1=C_0\left(\alpha_1+(-1)^{c+1}(\alpha_1^{b+c-1}-2\alpha_1^{a+c-1})\right).\]
If we choose $C_0=1+(-1)^c(\alpha_2^{b+c}-2\alpha_2^{a+c})$ then $f_1$ has to be a rational. Therefore let us compute
\begin{align*}
f_1=& \left(1+(\alpha_2^{b+c}-2\alpha_2^{a+c})(-1)^c\right)\left(\alpha_1+(\alpha_1^{b+c-1}-2\alpha_1^{b+c-1}) (-1)^{c+1}\right)\\
=&\alpha_1+\stackrel{R}{\overbrace{(\alpha_2^{b+c-1}+\alpha_2^{b+c-1})(-1)^{c+1}+(\alpha_1^{b+c-1}-2\alpha_1^{b+c-1}) (-1)^{c+1}}}\\
&-\stackrel{Q}{\overbrace{(\alpha_2^{b+c}-2\alpha_2^{a+c})(\alpha_1^{b+c-1}-2\alpha_1^{a+c-1})}}
\end{align*}
obviously $R$ is rational since it is the trace of an algebraic number. Let us consider $Q$ under the assumption $a-b=1$:
\begin{multline*}
(\alpha_2^{b+c}-2\alpha_2^{a+c})(\alpha_1^{b+c-1}-2\alpha_1^{a+c-1})\\=(-1)^{b+c-1}\alpha_2+4\alpha_2(-1)^{a+c+1}+
 -2(-1)^{b+c-1}\alpha_2^{a-b+1}-2(-1)^{a+c-1}\alpha_2^{b-a+1}\\
=(-1)^{b+c}(3\alpha_2+2\alpha_2^2-2).
\end{multline*}
Therefore
\[\alpha_1-(-1)^{b+c}(3\alpha_2+2\alpha_2^2-2)\]
has to be rational. Inserting the possibilities for $\alpha_1$ and
$\alpha_2$ we see that $\alpha_1=-1-\sqrt{2}$ if $b+c\equiv 0 \mod
2$ and $\alpha_1=-\frac{1+\sqrt{5}}2$ otherwise.

In the case $B=-2$ we obtain by a similar computation $\alpha_1=-1-\sqrt{2}$ if $b+c\equiv 1 \mod 2$ and $\alpha_1=-\frac{1+\sqrt{5}}2$ otherwise.

Inserting all possibilities we obtain exactly the sequences listed
in Table \ref{Tab:Sym}.

\subsection{One parameter family}
We may exclude the polynomial $2X^a-X^b-1$ from our considerations since the transformation $X\rightarrow 1/X$ yields the polynomial $-Y^a-Y^b+2=(-1)(Y^a+Y^b-2)$ with $Y=1/X$ and this equivalent to a transformation $f_n\rightarrow f_{-n}$. Now we have to consider which quadratic polynomials with no roots of unity in their set of roots divide $X^a-2X^b+1$ or $X^a+X^b-2$. The first polynomial was studied by Schinzel \cite{Schinzel:1962}:

\begin{lemma}[Schinzel]\label{Lem:Schinzel}
The polynomial
\[\frac{X^n-2X^m+1}{X^{\gcd(n,m)}-1}\]
is irreducible over $\Q$ for all $n>m>0$, except $n=7k$ and $m=5k$
or $m=2k$. Then the polynomial factors into
$$(X^{3k}+X^{2k}-1)(X^{3k}+X^k-1)$$ or
$$(X^{3k}+X^{2k}+1)(X^{3k}-X^k-1)$$.
\end{lemma}

Therefore we either have $a=3$ which yields polynomials listed in
Table \ref{Tab:Bin} or $a-\gcd(a,b)=2$ and $a>3$. Since
$a-\gcd(a,b)\geq a/2$ if $a>b$ we deduce $a=4$ and $b=2$ but
$X^4-2X^2+1=(X^2-1)^2$ which yields only degenerate or unitary
sequences.

We consider now the case $X^a+X^b-2$. Schinzel \cite{Schinzel:1965} proves in particular that the polynomial $\frac{X^a+X^b-2}{X^{\gcd(a,b)}-1}$ is irreducible if $a/\gcd(a,b)<C$ where $C$ is an absolute computable constant.
However the constant is by far too large ($>(10^8 !)^4$) to prove the irreducibility of $X^a+X^b-2$. Also the bound in \cite{Schinzel:1969} is too large ($\sim 10^8$). But following ideas of Schinzel \cite{Schinzel:1962} and Ljungreen \cite{Ljunggren:1960} we can show

\begin{lemma}\label{Lem:irr}
The polynomial
\[\frac{X^n+X^m-2}{X^{\gcd(n,m)}-1}\]
 is irreducible over $\Q$ for all $n>m>0$.
\end{lemma}

In order not to interrupt the proof of Theorem \ref{Th:Bin} we
postpone the proof of the lemma to the next subsection.

By Lemma \ref{Lem:irr} we have $a=3$ which only yields polynomials listed in Table \ref{Tab:Bin}, or we have $a=4$ and $b=2$. But
\[X^4+X^2-2=(X-1)(X+1)(X^2-2)\]
and hence each factor yields a degenerate or unitary recurrence. Therefore we have proved Theorem \ref{Th:Bin} completely apart form Lemma \ref{Lem:irr}.

\subsection{Proof of Lemma \ref{Lem:irr}}
Suppose that $f(X)g(X)=X^n+X^m-2$. Therefore we have $f(0)g(0)=-2$
and without loss of generality we have $|f(0)|=1$. On the other
hand a root of $X^m+X^n-2$ cannot have a root $\alpha$ such that
$|\alpha|<1$ since otherwise $|\alpha^n+\alpha^m-2|\geq
2-|\alpha|^n-|\alpha|^m>0$. Therefore all roots of $X^m+X^n-2$
have absolute value at least $1$. Since the product of all roots
of $f(X)$ has absolute value $1$ no root $\alpha$ has absolute
value greater than $1$, since then another root $\alpha'$ must
have absolute value less than $1$, i.e. all roots $\alpha$ of
$f(X)$ satisfy $|\alpha|=1$. Therefore we can write $\alpha=e^{2
\pi i y}$ and since $\alpha$ is a root of $X^m+X^n-2$ considering
real parts we obtain $\cos (2\pi m y)=\cos (2\pi n y)=1$, hence
$\sin (2\pi ny)=\sin (2 \pi my)=0$. Therefore $2y$ is rational and
its denominator divides $n$ and $m$. Therefore $\alpha$ is an
$\gcd(n,m)$-th  root of unity, i.e. $f(X)|
X^{\gcd(n,m)}-1$.\hfill$\blacksquare$

\subsection{Proof of Corollaries \ref{Cor:Fib} and \ref{Cor:unitary}}

By Theorem \ref{Th:nondom} and Theorem \ref{Th:Bin} we know that
the companion polynomial must be one of the polynomials listed in
Table \ref{Tab:Bin}. The root $\alpha$ with maximal absolute value
of all polynomials listed in Table \ref{Th:Bin} except the
polynomial $X^2-X-1$ is either not real or satisfies $\alpha<1$.
Therefore sequences with such companion polynomial are not
increasing. Therefore a sequence satisfiying the conditions of
Corollary \ref{Cor:Fib} has companion polynomial $X^2-X-1$. Since
a binary recurrence is uniquely determined by its companion
polynomial and the values at $f_0$ and $f_1$ the first part of
Corollary \ref{Cor:Fib} is proved.

It is well known (Binet's formula) that the Fibonacci sequence is defined by the explicit formula
\[f_n=\frac 1{\sqrt 5} \left( \left(\frac {1+\sqrt{5}}2\right)^n-\left(\frac {1-\sqrt{5}}2\right)^n\right).\]
Since $f_n$ is increasing a non-trivial three-term arithmetic
progression with $f_m>f_n>f_k$ fulfills $m>n>k\geq 0$ and in
particular for $k>0$ we have
\begin{multline*}
\left(\frac{1+\sqrt 5}2\right)^m +\left(\frac{1+\sqrt 5}2\right)^k-2\left(\frac{1+\sqrt 5}2\right)^n=\\
\left(\frac{1-\sqrt 5}2\right)^m +\left(\frac{1-\sqrt 5}2\right)^k-2\left(\frac{1-\sqrt 5}2\right)^n
\end{multline*}
Dividing this equation by $\left(\frac{1+\sqrt 5}2\right)^k$ and
using some simple estimations we obtain
\begin{equation}\label{Eq:CorFib}
\left|2\left(\frac{1+\sqrt 5}2\right)^{n-k}-\left(\frac{1+\sqrt
5}2\right)^{m-k}-1\right|<4\left(\frac{3-\sqrt 5}2\right)^k.
\end{equation}
Let us put $m-k=a$ and $n-k=b$ in \eqref{Eq:CorFib}, then we obtain
\[\left|\left(\frac{1+\sqrt 5}2\right)^b\left(\left(\frac{1+\sqrt 5}2 \right)^{a-b}-2\right)+1\right|<4\left(\frac{3-\sqrt 5}2\right)^k.\]

First, let us consider the case $a-b=1$. We may exclude $b=2$ since this yields the known family of three-term arithmetic progressions.
With $a-b=1$ we obtain
\[\left|1-\left(\frac{1+\sqrt 5}2\right)^{b-2}\right|<4\left(\frac{3-\sqrt 5}2\right)^k.\]
For $b=1$ and $b=3$ we obtain $k\leq 2$ and for $b>3$ we have $k=0$ a contradiction. Therefore the only possibilities for the triple $(m,n,k)$ are:
\[(3,2,1), \;\; (4,3,2), \;\; (5,4,1), \; \text{and} \; (6,5,2).\]
The only new triple that indeed provides an arithmetic progression is $(4,3,2)$, but this triple is listed in the corollary. For $a-b=2$ we obtain the inequality
\[\left|1-\left(\frac{1+\sqrt 5}2\right)^{b-1}\right|<4\left(\frac{3-\sqrt 5}2\right)^k.\]
For $b=2$ we obtain $k\leq 2$ and for $b>2$ we have $k=0$, a contradiction. In the case $b=1$ the original equation turns into
\begin{multline*}
\left(\frac{1+\sqrt 5}2\right)^{k+3} +\left(\frac{1+\sqrt 5}2\right)^k-2\left(\frac{1+\sqrt 5}2\right)^{k+1}=\\
\left(\frac{1-\sqrt 5}2\right)^{k+3} +\left(\frac{1-\sqrt 5}2\right)^k-2\left(\frac{1-\sqrt 5}2\right)^{k+1}.
\end{multline*}
Dividing through $\left(\frac{1+\sqrt 5}2\right)^k$ and a simple estimation on the right side yields
\[2=\left(\frac{1+\sqrt 5}2\right)^3 +1-2\cdot\frac{1+\sqrt 5}2 < 4\left|\frac{1-\sqrt 5}2 \right|^{2k}.\]
Therefore we have $k=0$, also a contradiction. Therefore the triple $(m,n,k)$ is either $(5,3,1)$ or $(6,4,2)$. But both do not yield arithmetic progressions. In the case of $a-b\geq 3$ the left side of \eqref{Eq:CorFib} is at least $\frac{7+\sqrt 5}2>4$ which is larger than the right side,
hence this case does not occur.

Now we consider the case $k=0$. Then inequality \eqref{Eq:CorFib} turns into
\[\left|\left(\frac {1+\sqrt 5}2\right)^{m-n}-2\right|<3 \left(\frac{1-\sqrt 5}2\right)^{2n}.\]
If $m-n=1$ we obtain
\[\left|\frac{3-\sqrt 5}2\right|<3  \left(\frac{3-\sqrt 5}2\right)^{n},\]
i.e. we have $n=1,2$. For $m-n=2$ we have
\[\left|\frac{1-\sqrt 5}2\right|<3  \left(\frac{3-\sqrt 5}2\right)^{n},\]
i.e. we have $n=1$. For $m-n>2$ we have
\[\sqrt 5\leq \left|\left(\frac {1+\sqrt 5}2\right)^{m-n}-2\right|<3  \left(\frac{3-\sqrt 5}2\right)^{n},\]
i.e. $n=0$, a contradiction. Therefore we have for the triple $(m,n,k)$ only the possibilities
$(2,1,0),(3,2,0)$ or $(3,1,0)$ which all appear in the statement of the corollary or yield no arithmetic progression.

We have seen that an arithmetic three-term progression in the Fibonacci sequence comes from the infinite family or has highest index at most $4$. Therefore an arithmetic four-term progression might have highest index at most $5$. Writing down $f_n$ for $n=0, \ldots,5$ we see that the only four-term arithmetic progressions appearing are those written down in the corollary.

Now we turn to the proof of Corollary \ref{Cor:unitary}. A unitary
increasing binary recurrence defined over the rationals is of the
form $f_n=ca^n+d(\pm 1)^n$, where $1<a\in\Q$ and $c,d\in\Q$. In
the case of $+1$ also the sequence $a^n$ yields infinitely many
three-term arithmetic progression. But in this case we have
$a^m-2a^n+a^k=0$ or $a^x-2a^y+1=0$, with $x>y>0$. On the other
hand we claim $a^x-2a^y+1\neq 0$ if $a\neq \pm 1$, but $a=\pm 1$
yields $f_n$ degenerate or constant. For the proof of the claim
see at the end of this section. Therefore we assume the $-1$ case.
Note that the equation $2f_n=f_m+f_k$ with $m>n>k\geq 0$ turns
into
\begin{equation}\label{Eq:unitary} a^m-2a^n+a^k=-\frac dc ((-1)^m-2(-1)^n+(-1)^k)< C_0\end{equation}
where $C_0$ is an absolute constant. If $a>2$ the left side tends
to $\infty$ whenever $m\rightarrow \infty$. Therefore let us
consider the case $2>a>1$. If $a^k<C_0$ the inequality above turns
into $|a^m-2a^n|<2C_0$ and in the case of $a^k\geq C_0$ we divide
by $a^k$ and obtain $|a^{m-k}-2a^{n-k}|<1+C_0/a^k\leq 2$. In any
case there is a constant $C_1$ such that
\[|a^y-2a^x|<C_1,\]
where $y=m,m-k$ and $x=n,n-k$ depending on the size of $a^k$. By a
standard application of Baker's theory of linear forms in
logarithms (see e.g. \cite{Baker:1993}) the inequality has only
finitely many solutions $x$ and $y$. Note that $a$ and $2$ are
multiplicatively independent. For small $k$ we therefore deduce
that only finitely many solutions to \eqref{Eq:unitary} exist. We
claim that $a^x-2a^y+1\not=0$ for $2>a>1$ with $a\in\Q$. If this
is true, then since $m-k$ and $n-k$ obtain only finitely many
values we find a constant $C_2$ such that
\[C_0<a^m-2a^n+a^k=(a^{m-k}-2a^{n-k}+1)a^k<C_2a^k.\]
Hence $k$ takes only finitely many values, i.e. \eqref{Eq:unitary} has only finitely many solutions. Therefore we deduce $a=2$ and $f_n=c2^n+d(-1)^n$.
Since $f_0=0$ and $f_1=1$ we deduce $f_n=\frac{2^n-(-1)^n}3$
which fulfills all requirements of the corollary.

So we are left to prove the
claim $a^x-2a^y+1\not=0$ for $2>a>1$ with $a\in \Q$.
But by Lemma \ref{Lem:Schinzel} we know that the only linear factors of the
polynomial $X^n-2X^m+1$ may only have the linear factors $(X-1)$ and $(X+1)$.

\section{The ternary case}

To obtain infinite families in the symmetric case or in the exceptional case the degree of the companion polynomial has to be even an obvious contradiction. So only the case of one parameter families remains.

First, let us consider factors of degree $3$ of $X^a-2X^b+1$. Due to Schinzel's result (see Lemma \ref{Lem:Schinzel}) we may assume $a=4,5,6,7$. In case of $a=4$ we may choose $b=1$ or $b=3$ (for $b=2$ the non-cyclotomic factor would be of degree $\leq 2$).
The corresponding irreducible non-cyclotomic factors of degree $3$ are
\[X^3-X^2-X-1 \;\; \text{and} \;\; X^3+X^2+X-1.\]
If $a=5$ we have no factor of degree $3$. In the case of $a=6$ we may choose $b=3$ but then we have $X^6-2X^3+1=(X^3-1)^2$ which has only cyclotomic factors. In the case of $a=7$ we may choose $b=2$ or $b=5$ and the non-cyclotomic factors are for $b=2$
\[ X^3+X^2+1 \;\; \text{and} \;\; X^3-X-1\]
and for $b=5$
\[X^3+X^2-1 \;\; \text{and} \;\; X^3+X-1.\]
So all polynomials listed in Theorem \ref{Th:Ternary} were found and there are no further possibilities left.

Now we consider factors of degree $3$ of $X^a+X^b-2$. Due to Lemma \ref{Lem:irr} we know that $a=4,5,6$. Similar as above we may exclude $a=5$. In the case of $a=6$ the non-cyclotomic factor of degree $3$ is $X^3-2$ but this polynomial only yields degenerate recurrences. So we are left to the case $a=4$ and $b=1$ or $b=3$. Therefore we obtain the polynomials
\[X^3+X^2+X+2 \;\; \text{and} \;\; X^3+2X^2+2X+2.\]

By replacing $X$ by $1/X$ and expanding, the case $2X^a-X^b-1$ is equivalent to the case $X^a+X^b-2$. Therefore we obtain $2$ further polynomials namely
\[2X^3+X^2+X+1 \;\; \text{and} \;\; 2X^3+2X^2+2X+1.\]
\hfill $\blacksquare$

{\bf Acknowledgements.} The authors are grateful to Professor
Andrzej Schinzel for his helpful remarks, especially for the proof
of Lemma \ref{Lem:irr} which essentially simplifies our original
proof and to Professor Jan-Hendrik Evertse for reading this
manuscript and his comments.

\bibliographystyle{abbrv}
\bibliography{Fibon}
\end{document}